\documentclass[11pt]{article}
\usepackage{}
\usepackage{graphicx}
\usepackage{amsmath}
\usepackage{mathrsfs}
\usepackage{amsfonts}
\usepackage{amssymb}
\usepackage[noadjust]{cite}
\usepackage{enumerate}
\usepackage{latexsym}
\usepackage{graphpap}
\usepackage{color} 
\usepackage[cm]{fullpage}
\usepackage[english]{babel}
\usepackage[latin1]{inputenc}
\usepackage{amsthm}

\newtheorem{theorem}{Theorem}[section]
\newtheorem{corollary}[theorem]{Corollary}
\newtheorem{lemma}[theorem]{Lemma}
\newtheorem{proposition}[theorem]{Proposition}

\numberwithin{equation}{section}

\def\im{\mathop{\mathrm{Im}}\nolimits}
\def\rank{\mathop{\mathrm{rank}}\nolimits}
\def\dom{\mathop{\mathrm{Dom}}\nolimits}
\def\fix{\mathop{\mathrm{Fix}}\nolimits}

\newcommand{\transf}[1]{\left(\begin{smallmatrix}#1\end{smallmatrix}\right)} 

\newcommand{\lastpage}{\addresss}

\newcommand{\addresss}{\small \sf  

\noindent{\sc De Biao Li}, 
School of Mathematics, 
North University of China, 
Taiyuan, Shanxi 030000, 
P. R. China. \\
e-mail: 20230056@nuc.edu.cn

\medskip

\noindent{\sc V\'\i tor H. Fernandes}, 
Center for Mathematics and Applications (NOVA Math) 
and Department of Mathematics, NOVA FCT, 
Faculdade de Ci\^encias e Tecnologia, 
Universidade Nova de Lisboa, 
Monte da Caparica, 
2829-516 Caparica, 
Portugal. \\
e-mail: vhf@fct.unl.pt
}

\author{
De Biao Li\footnote{This work is partially supported by the National Natural Science Foundation of China (Nos.12271224, 12171213).}
~and V\'\i tor H. Fernandes\footnote{This work is funded by national funds through the FCT - Funda\c c\~ao para a Ci\^encia e a Tecnologia, 
I.P., under the scope of the projects UIDB/00297/2020 and UIDP/00297/2020 (NOVA Math - Center for Mathematics and Applications).}
}

\title{On semigroups of orientation-preserving partial permutations with restricted range}

\begin{document}
%=========================================================================

\maketitle 

%========================================================================
%========================================================================
\begin{abstract}
Let $\Omega_n$ be a finite chain with $n$ elements $(n\in\mathbb{N})$, and let $\mathcal{POPI}_{n}$ be the semigroup of all injective orientation-preserving partial transformations of $\Omega_n$. In this paper, for any nonempty subset $Y$ of $\Omega_n$, 
we consider the subsemigroup $\mathcal{POPI}_{n}(Y)$ of $\mathcal{POPI}_{n}$ of all transformations with range contained in $Y$. 
We describe the Green's relations and study the regularity of $\mathcal{POPI}_{n}(Y)$. 
Moreover, we calculate the rank of $\mathcal{POPI}_{n}(Y)$ and 
determine when two semigroups of this type are isomorphic.
\end{abstract}

\medskip



%========================================================================
\section*{Introduction and preliminaries} 

Let $\Omega_n=\{1, 2,\ldots, n\}$ be an $n$-element set ordered in the standard way.
We denote by $\mathcal{PT}_{n}$ the monoid (under composition) of all partial transformations on $\Omega_n$, by $\mathcal{T}_{n}$ the submonoid of $\mathcal{PT}_{n}$ of all full transformations on $\Omega_n$, by $\mathcal{I}_{n}$ the \textit{symmetric inverse monoid} on $\Omega_n$, i.e. the inverse submonoid of $\mathcal{PT}_{n}$ of all partial permutations on $\Omega_n$, 
and by $\mathcal{S}_{n}$ the \textit{symmetric group} on $\Omega_n$, 
i.e. the subgroup of $\mathcal{PT}_{n}$ of all permutations on $\Omega_n$.

A transformation $\alpha$ in $\mathcal{PT}_{n}$ is called \textit{order-preserving} 
if $x\leqslant y$ implies $x\alpha\leqslant y\alpha$, 
for all $x,y \in \dom(\alpha)$. We denote by $\mathcal{O}_{n}$ the submonoid of $\mathcal{T}_{n}$ of all order-preserving 
full transformations and by $\mathcal{POI}_{n}$ the submonoid of $\mathcal{I}_{n}$ of all order-preserving partial permutations. 

Let $s=(a_1,a_2,\ldots,a_t)$ be a sequence of $t$ ($t\geqslant0$) elements from the chain $\Omega_n$. 
We say that $s$ is \textit{cyclic} if there exists no more than one index $i\in\{1,\ldots,t\}$ such that
$a_i>a_{i+1}$, where $a_{t+1}$ denotes $a_1$.
Notice that, the sequence $s$ is cyclic if and only if $s$ is empty 
or there exists $i\in\{0,1,\ldots,t-1\}$ such that $a_{i+1}\leqslant a_{i+2}\leqslant \cdots\leqslant a_t\leqslant a_1\leqslant \cdots\leqslant a_i $ 
(the index $i\in\{0,1,\ldots,t-1\}$ is unique unless $s$ is constant and $t\geqslant2$). 
Given a partial transformation $\alpha\in\mathcal{PT}_{n}$ such that
$\dom(\alpha)=\{a_1<\cdots<a_t\}$, with $t\geqslant0$, we 
say that $\alpha$ is \textit{orientation-preserving}  if the sequence of its images
$(a_1\alpha,\ldots,a_t\alpha)$ is cyclic.   
We denote by $\mathcal{OP}_{n}$ and $\mathcal{POPI}_{n}$ the submonoids of $\mathcal{T}_{n}$ 
and $\mathcal{I}_{n}$ of all orientation-preserving
transformations, respectively.

As usual, the \textit{rank} of a finite semigroup $S$ is defined by
$
\rank(S)=\min\{|A|\mid A\subseteq S,\,\langle A\rangle=S\}.
$ 
It is well known that, for $n\geqslant3$, the ranks of $\mathcal{PT}_{n}$, $\mathcal{T}_{n}$, $\mathcal{I}_{n}$ and $\mathcal{S}_{n}$ are equal to $4$, $3$, $3$ and $2$, respectively. Gomes and Howie~\cite{Gomes&Howie:1992} showed that the rank of the monoid $\mathcal{O}_{n}$ is equal to $n$. Catarino~\cite{Catarino} established the rank of the monoid $\mathcal{OP}_{n}$, and in~\cite{Catarino&Higgins}, Catarino and Higgins proved that $\mathcal{OP}_{n}$ is a regular submonoid of $\mathcal{T}_{n}$. In~\cite{Fernandes&al:2009}, Fernandes et al. described the congruences on $\mathcal{OP}_{n}$ and, 
in~\cite{Fernandes:2001}, Fernandes showed that the rank of the monoid $\mathcal{POI}_{n}$ is equal to $n$ and established a presentation for this monoid.
The classification of maximal subsemigroups of $\mathcal{POI}_{n}$ was obtained by Ganyushkin and Mazorchuk~\cite{Ganyushkin&Mazorchuk:2003}. In~\cite{Fernandes:2000}, Fernandes proved that the rank of $\mathcal{POPI}_{n}$ is equal to $2$, described the congruences and gave a presentation for this monoid. 

\smallskip 

Given a nonempty subset $Y$ of $\Omega_n$, let
\begin{align*}
\mathcal{T}_{n}(Y)&=\{\alpha\in\mathcal{T}_{n}\mid\im(\alpha)\subseteq Y\},\\
\mathcal{PT}_{n}(Y)&=\{\alpha\in\mathcal{PT}_{n}\mid\im(\alpha)\subseteq Y\},\\
\mathcal{I}_{n}(Y)&=\{\alpha\in\mathcal{I}_{n}\mid\im(\alpha)\subseteq Y\},\\
\mathcal{O}_{n}(Y)&=\{\alpha\in\mathcal{O}_{n}\mid\im(\alpha)\subseteq Y\},\\
\mathcal{OP}_{n}(Y)&=\{\alpha\in\mathcal{OP}_{n}\mid\im(\alpha)\subseteq Y\},\\
\mathcal{POI}_{n}(Y)&=\{\alpha\in\mathcal{POI}_{n}\mid\im(\alpha)\subseteq Y\}.
\end{align*}

In 1975, Symons~\cite{Symons:1975} introduced and studied the semigroup $\mathcal{T}_{n}(Y)$. He described all the automorphisms of $\mathcal{T}_{n}(Y)$ and also determined when two semigroups of this type are isomorphic. In~\cite{Nenthein&al:2005}, Nenthein et al. characterized the regular elements of $\mathcal{T}_{n}(Y)$ and, in~\cite{Sanwong&Sommanee:2008}, Sanwong and Sommanee obtained the largest regular subsemigroup of 
$\mathcal{T}_{n}(Y)$ and they also determined a class of maximal inverse subsemigroups of this semigroup. Later, in 2009, all maximal and minimal congruences on $\mathcal{T}_{n}(Y)$ were described by Sanwong et al.~\cite{Sanwong&al:2009}. On the other hand, Sullivan~\cite{Sullivan:2008} consider the linear counterpart of $\mathcal{T}_{n}(Y)$, that is the semigroup which consists of all linear transformations from a vector space $V$ into a fixed subspace $W$ of $V$, and described its Green's relations and ideals.
The rank of the semigroups $\mathcal{T}_{n}(Y)$, $\mathcal{PT}_{n}(Y)$ and $\mathcal{I}_{n}(Y)$ were considered by Fernandes and Sanwong in~\cite{Fernandes&Sanwong:2014}. A description of the regular elements of $\mathcal{O}_{n}(Y)$ was given by Mora and Kemprasit in~\cite{Mora&Kemprasit:2010}. This semigroup was also studied by Fernandes et al. in~\cite{Fernandes&al:2014} who described its largest regular subsemigroup and Green's relations. 
Moreover, also in~\cite{Fernandes&al:2014}, Fernandes et al. determined when two semigroups of the type $\mathcal{O}_{n}(Y)$ are isomorphic and calculated their ranks. Fernandes et al.~\cite{Fernandes&al:2016} describe the largest regular subsemigroup of the semigroup $\mathcal{OP}_{n}(Y)$, they also determine when two semigroups of type $\mathcal{OP}_{n}(Y)$ are isomorphic and calculated the rank of the semigroup $\mathcal{OP}_{n}(Y)$. Recently, in~\cite{Li&al:2023}, Li et al. considered the rank properties of the semigroup $\mathcal{POI}_{n}(Y)$ and they also determined when two semigroups of the type 
$\mathcal{POI}_{n}(Y)$ are isomorphic. For more results about semigroups of transformations with restricted range, for example, see~\cite{Dimitrova&Koppitz:2022,Dimitrova&Koppitz:2021,Tinpun&Koppitz:2016a,Tinpun&Koppitz:2016b}.

\smallskip 

In this paper, for each nonempty subset $Y$ of $\Omega_n$,
we consider the semigroup of transformations with restricted range
$$
\mathcal{POPI}_{n}(Y)=\{\alpha\in\mathcal{POPI}_{n}\mid\im(\alpha)\subseteq Y\}.
$$

This paper is organized as follows. Some preliminary knowledge and notation are given at the end of this section.  
Section~\ref{secb} is dedicated to the study of regularity and Green's relations on $\mathcal{POPI}_{n}(Y)$. 
In Section~\ref{secc},  we determine the cardinality of the semigroup $\mathcal{POPI}_{n}(Y)$ and when two semigroups of the type $\mathcal{POPI}_{n}(Y)$ are isomorphic. In Section~\ref{secd}, we investigate the rank of $\mathcal{POPI}_{n}(Y)$. 

Throughout this paper, $Y$ will always denote a nonempty subset of $\Omega_n$.  

\medskip 

Let $S$ be a semigroup. Denote by $S^{1}$ the monoid obtained from $S$ through the adjoining of an identity if necessary. Recall that the definitions of Green's equivalences $\mathscr{L}, \mathscr{R}, \mathscr{J}, \mathscr{H}$ and $\mathscr{D}$ are defined by: for all $a, b\in S$,
\begin{align*}
(a,b)\in\mathscr {L} \quad &\mbox{if and only if} \quad S^{1}a=S^{1}b,\\
(a,b)\in\mathscr {R} \quad &\mbox{if and only if} \quad aS^{1}=bS^{1},\\
(a,b)\in\mathscr {J} \quad &\mbox{if and only if} \quad S^{1}aS^{1}=S^{1}bS^{1},
\end{align*}
$\mathscr{H}=\mathscr{L}\cap\mathscr{R}$ and $\mathscr{D}=\mathscr{L}\vee\mathscr{R}=\mathscr{L}\circ\mathscr{R}=\mathscr {R}\circ\mathscr{L}$. It is well known that $\mathscr{D}=\mathscr{J}$ in a finite semigroup.  For $\mathscr{K}\in\{\mathscr{L}, \mathscr{R}, \mathscr{J}, \mathscr{H}, \mathscr{D}\}$ and $a\in S$, we denote 
the $\mathscr{K}$-class containing $a$ by $\mathscr{K}_{a}$. 
If necessary, to avoid ambiguity, we also denote 
the Green's equivalence $\mathscr{K}$ and the $\mathscr{K}$-class containing $a$ on the semigroup $S$ by $\mathscr{K}^{S}$  and by $\mathscr{K}^S_{a}$, respectively. 

An element $a$ of $S$ is called \textit{regular} if there exists $b$ in $S$ such that $aba=a$. 
The semigroup $S$ is said to be \textit{regular} if all its elements are regular.

\smallskip 

We refer the reader to the monograph of Howie~\cite{Howie:1995} for any undefined notation and terminology.

\section{Regularity and Green's relations}\label{secb}

In this section, we characterize the regular elements of $\mathcal{POPI}_{n}(Y)$ and Green's relations on 
$\mathcal{POPI}_{n}(Y)$. In addition, we determine when $\mathcal{POPI}_{n}(Y)$ is a regular semigroup. 

Let $\alpha\in \mathcal {PT}_{n}$. We denote by $\dom(\alpha)$ and $\im(\alpha)$ the \textit{domain} and \textit{image} of $\alpha$, respectively. 
Define also $\rank(\alpha)=|\im(\alpha)|$. 

\begin{proposition}\label{regular}
Let $\alpha\in \mathcal{POPI}_{n}(Y)$. Then, the following are equivalent:
\begin{enumerate}[(i)]
\item $\alpha$ is regular;

\item $\im(\alpha)=Y\alpha$;

\item $\dom(\alpha)\subseteq Y$.
\end{enumerate}
\end{proposition}

\begin{proof}
(i)\,$\Rightarrow$\,(ii) Suppose that $\alpha\in\mathcal{POPI}_{n}(Y)$ is regular. Then, there exists
 some $\beta\in\mathcal{POPI}_{n}(Y)$ such that $\alpha=\alpha\beta\alpha$, and so ${\im}(\alpha)=\Omega_n\alpha=(\Omega_n\alpha\beta)\alpha\subseteq Y\alpha\subseteq \Omega_n\alpha=\im(\alpha)$. Hence, ${\im}(\alpha)=Y\alpha$.

(ii)\,$\Rightarrow$\,(iii) Since ${\im}(\alpha)=Y\alpha$ and $\alpha$ is injective, it follows that ${\dom}(\alpha)\subseteq Y$.

(iii)\,$\Rightarrow$\,(i) Suppose that ${\dom}(\alpha)\subseteq Y$. If $\dom(\alpha)=\emptyset$, then $\alpha=\varnothing$ and so it is clearly regular.  If $\dom(\alpha)\neq\emptyset$, then ${\im}(\alpha) \neq\emptyset$. Let ${\im}(\alpha)=\{z_{1}<\cdots<z_{m}\}\subseteq Y$. Then, there exist $y_{1},\ldots,y_{m}\in Y$ such that $y_{i}\alpha=z_{i}$ for $1\leqslant i\leqslant m$, and the sequence $(y_{1},y_{2},\ldots,y_{m})$ is cyclic.
Let
$
\beta=\transf{z_{1} &z_{2} &\cdots &z_{m}\\y_{1} &y_{2} &\cdots &y_{m}}.
$
Clearly, $\beta\in\mathcal{POPI}_{n}(Y)$. Then, it is routine to verify that $\alpha=\alpha\beta\alpha$, whence $\alpha$ is regular.
\end{proof}

The next result determines when $\mathcal{POPI}_{n}(Y)$ is a regular semigroup.

\begin{corollary}\label{coro:regular}
Let $Y$ be a nonempty subset of $\Omega_n$. Then, $\mathcal{POPI}_{n}(Y)$ is regular if and only if $Y=\Omega_n$.
\end{corollary}

\begin{proof}
If $Y=\Omega_n$, then for any $\alpha\in \mathcal{POPI}_{n}(Y)$, ${\dom}(\alpha)\subseteq Y$. It follows from Proposition~\ref{regular} that $\alpha$ is regular, whence, $\mathcal{POPI}_{n}(Y)$ is regular (notice that, in~\cite{Fernandes:2000}, it was already proved that $\mathcal{POPI}_{n}$ is regular). 
If $Y \ne \Omega_n$, then there exist elements $x\in \Omega_n\backslash Y$ and $y\in Y$. Clearly, $\alpha=\transf{x\\y}\in \mathcal{POPI}_{n}(Y)$ and 
${\dom}(\alpha)\nsubseteq Y$, whence $\alpha$ is not regular in $\mathcal{POPI}_{n}(Y)$, by Proposition~\ref{regular}. Thus, $\mathcal{POPI}_{n}(Y)$ is not regular.
\end{proof}

Recall the following result from \cite{Fernandes:2000}:

\begin{lemma}\cite[Propositions 2.4, 2.5 and 2.6]{Fernandes:2000} \label{GPOPI}
Let $\alpha,\beta\in \mathcal{POPI}_{n}$. Then: 
\begin{enumerate}[(i)]
\item $(\alpha, \beta)\in\mathscr{L}$  if and only if $\im(\alpha)=\im(\beta)$; 
\item $(\alpha, \beta)\in\mathscr{R}$  if and only if $\dom(\alpha)=\dom(\beta)$;
\item Let $\alpha\in\mathcal{POPI}_{n}$ be such that $1\leq\rank(\alpha)=m\leqslant n$. Then, $|\mathscr{H}_{\alpha}|=m$. Moreover, if $\alpha$ is an idempotent, then $\mathscr{H}_{\alpha}$ is a cyclic group of order $m$;
\item $(\alpha, \beta)\in\mathscr{D}$  if and only if $\rank(\alpha)=\rank(\beta)$.
\end{enumerate}
\end{lemma}

The following theorem describes the Green's relations $\mathscr{L}$, $\mathscr{R}$ and $\mathscr{H}$ on $\mathcal{POPI}_{n}(Y)$. 
Let us denote the partial identity transformation on a subset $A$ of $\Omega_n$ by $\mathrm{id}_{A}$.  
Clearly, if $A\subseteq Y$, then $\mathrm{id}_{A}\in\mathcal{POPI}_{n}(Y)$

\begin{theorem}\label{green}
Let $\alpha,\beta\in \mathcal{POPI}_{n}(Y)$. Then
\begin{enumerate}[(i)]
\item $(\alpha, \beta)\in\mathscr{L}$  if and only if either both $\alpha$ and $\beta$ are regular and $\im(\alpha)=\im(\beta)$, or $\alpha=\beta$.
\item $(\alpha, \beta)\in\mathscr{R}$  if and only if $\dom(\alpha)=\dom(\beta)$.
\item Let $\alpha\in\mathcal{POPI}_{n}(Y)$ be such that $\rank(\alpha)=m$. If $\alpha$ is not regular, then $|\mathscr{H}_{\alpha}|=1$. If $\alpha$ is regular, then $|\mathscr{H}_{\alpha}|=m$ and all the elements of $\mathscr{H}_{\alpha}$ have the same domain and image. Moreover, if $\alpha$ is an idempotent, then $\mathscr{H}_{\alpha}$ is a cyclic group of order $m$.
\end{enumerate}
\end{theorem}

\begin{proof}
(i) Suppose that $(\alpha, \beta)\in\mathscr{L}^{\mathcal{POPI}_{n}(Y)}$. Then, 
there exist $\gamma, \delta\in\mathcal{POPI}_{n}(Y)^{1}$ such that $\alpha=\gamma\beta$ and $\beta=\delta\alpha$. 
Suppose that $\alpha$ is regular. If $\gamma=1$ or $\delta=1$, then $\alpha=\beta$, and so both $\alpha$ and $\beta$ are regular and $\im(\alpha)=\im(\beta)$; if $\gamma,\delta\neq1$, then $\Omega_n\alpha=\Omega_n\gamma\beta\subseteq Y\beta\subseteq \Omega_n\beta=\Omega_n\delta\alpha\subseteq Y\alpha\subseteq \Omega_n\alpha$ and so ${\im}(\alpha)=\Omega_n\alpha=Y\alpha=Y\beta=\Omega_n\beta=\im(\beta)$. Hence, both $\alpha$ and $\beta$ are regular and $\im(\alpha)=\im(\beta)$. Suppose that $\alpha$ is not regular. If $\gamma,\delta\neq1$, then $\Omega_n\alpha=\Omega_n\gamma\beta=\Omega_n\gamma\delta\alpha\subseteq Y\alpha\subseteq \Omega_n\alpha$ and so ${\im}(\alpha)=\Omega_n\alpha=Y\alpha$. Hence, $\alpha$ is regular by Theorem~\ref{regular}, a contradiction. Hence, $\gamma=1$ or $\delta=1$, and so $\alpha=\beta$.

Conversely, if $\alpha=\beta$, then trivially $(\alpha, \beta)\in\mathscr{L}^{\mathcal{POPI}_{n}(Y)}$. 
So, suppose that $\alpha$ and $\beta$ are regular and $\im(\alpha)=\im(\beta)$. 
Then, in particular, $(\alpha, \beta)\in\mathscr {L}^{\mathcal{POPI}_{n}}$, by Lemma~\ref{GPOPI}(i). 
Let $\gamma, \delta\in\mathcal{POPI}_{n}$ be such that $\alpha=\gamma\beta$ and $\beta=\delta\alpha$. 
Let $\alpha', \beta'\in\mathcal{POPI}_{n}(Y)$ be such that $\alpha=\alpha\alpha'\alpha$ and $\beta=\beta\beta'\beta$. 
Then $\alpha=\gamma\beta=\gamma\beta\beta'\beta$ and, clearly,  $\gamma\beta\beta'\in \mathcal{POPI}_{n}(Y)$. 
Similarly, $\beta=\delta\alpha\alpha'\alpha$ and $\delta\alpha\alpha'\in\mathcal{POPI}_{n}(Y)$. 
Thus, $(\alpha, \beta)\in\mathscr{L}^{\mathcal{POPI}_{n}(Y)}$.

\smallskip 

(ii) If $(\alpha, \beta)\in\mathscr{R}^{\mathcal{POPI}_{n}(Y)}$, then $(\alpha, \beta)\in\mathscr{R}^{\mathcal{POPI}_{n}}$, and so 
so $\dom(\alpha)=\dom(\beta)$, by Lemma~\ref{GPOPI}(ii). 

Conversely, suppose that $\dom(\alpha)=\dom(\beta)$. Then, $(\alpha, \beta)\in\mathscr{R}^{\mathcal{POPI}_{n}}$, by Lemma~\ref{GPOPI}(ii). 
Hence, $\alpha=\beta\gamma$ and $\beta=\alpha\delta$ for some $\gamma, \delta\in\mathcal{POPI}_{n}$. 
Clearly, $\alpha=\alpha\mathrm{id}_{\im(\alpha)}$ and $\beta=\beta\mathrm{id}_{\im(\beta)}$, whence 
$\alpha=\beta\gamma\mathrm{id}_{\im(\alpha)}$ and $\beta=\alpha\delta\mathrm{id}_{\im(\beta)}$. 
Since, clearly,  $\gamma\mathrm{id}_{\im(\alpha)},\delta\mathrm{id}_{\im(\beta)}\in\mathcal{POPI}_{n}(Y)$, 
we get $(\alpha, \beta)\in\mathscr{R}^{\mathcal{POPI}_{n}(Y)}$. 

\smallskip 

(iii) Let $\alpha\in\mathcal{POPI}_{n}(Y)$ be such that $\rank(\alpha)=m$. 
If $\alpha$ is not regular, then by (i) we have $|\mathscr{L}_{\alpha}|=1$, and so $|\mathscr{H}_{\alpha}|=1$. 
Therefore, suppose that $\alpha$ is regular. 
Then, for any $\gamma\in\mathscr{H}^{\mathcal{POPI}_{n}}_{\alpha}$, by Proposition~\ref{regular} and Lemma~\ref{GPOPI}, 
we have $\dom(\gamma)=\dom(\alpha)\subseteq Y$ and ${\im}(\gamma)={\im}(\alpha)\subseteq Y$. 
Hence, $\gamma\in\mathcal{POPI}_{n}(Y)$ and $\gamma\in\mathscr{H}^{\mathcal{POPI}_{n}(Y)}_{\alpha}$, by (i) and (ii). 
So $\mathscr{H}^{\mathcal{POPI}_{n}}_{\alpha}\subseteq\mathscr{H}^{\mathcal{POPI}_{n}(Y)}_{\alpha}$ and, 
as trivially $\mathscr{H}^{\mathcal{POPI}_{n}(Y)}_{\alpha}\subseteq\mathscr{H}^{\mathcal{POPI}_{n}}_{\alpha}$, 
we get $\mathscr{H}^{\mathcal{POPI}_{n}(Y)}_{\alpha}=\mathscr{H}^{\mathcal{POPI}_{n}}_{\alpha}$. Thus, by Lemma~\ref{GPOPI}(iii),  
$|\mathscr{H}^{\mathcal{POPI}_{n}(Y)}_{\alpha}|=m$, all the elements of $\mathscr{H}^{\mathcal{POPI}_{n}(Y)}_{\alpha}$ have the same domain and image, and 
if $\alpha$ is an idempotent, then $\mathscr{H}^{\mathcal{POPI}_{n}(Y)}_{\alpha}$ is a cyclic group of order $m$.
\end{proof}

Our next theorem gives a characterization of the Green's relation $\mathscr{D}=\mathscr{J}$ on $\mathcal{POPI}_{n}(Y)$.

\begin{theorem}\label{green D}
Let $\alpha,\beta\in \mathcal{POPI}_{n}(Y)$. Then, $(\alpha, \beta)\in\mathscr{D}$  if and only if either (1) both $\alpha$ and $\beta$ are regular and $\rank(\alpha)=\rank(\beta)$  or (2) both $\alpha$ and $\beta$ are not regular and $\dom(\alpha)=\dom(\beta)$.
\end{theorem}

\begin{proof}
Suppose that $(\alpha, \beta)\in\mathscr{D}^{\mathcal{POPI}_{n}(Y)}$. Then, $(\alpha, \beta)\in\mathscr{D}^{\mathcal{POPI}_{n}}$, and so, 
by Lemma~\ref{GPOPI}(iv), we get $\rank(\alpha)=\rank(\beta)$. 
Since any two $\mathscr{D}$-related elements (in any semigroup) are both regular or both not regular, 
then either both $\alpha$ and $\beta$ are regular or both $\alpha$ and $\beta$ are not regular. 
If both $\alpha$ and $\beta$ are regular, then we have (1). 
So, suppose that both $\alpha$ and $\beta$ are not regular and let 
$\gamma\in\mathcal{POPI}_{n}(Y)$ be such that $(\alpha, \gamma)\in\mathscr {L}$ and $(\gamma, \beta)\in\mathscr{R}$. 
Hence, by Theorem~\ref{green}, it follows that $\alpha=\gamma$ and $\dom(\gamma)=\dom(\beta)$. Thus, $\dom(\alpha)=\dom(\beta)$.

Conversely, if both $\alpha$ and $\beta$ are not regular and $\dom(\alpha)=\dom(\beta)$, then $(\alpha, \beta)\in\mathscr{R}^{\mathcal{POPI}_{n}(Y)}$, 
by Theorem~\ref{green}(ii), and so $(\alpha, \beta)\in\mathscr{D}^{\mathcal{POPI}_{n}(Y)}$. 
On the other hand, suppose both $\alpha$ and $\beta$ are regular and ${\rank}(\alpha)={\rank}(\beta)$. 
Then, $(\alpha, \beta)\in\mathscr{D}^{\mathcal{POPI}_{n}}$, by Lemma~\ref{GPOPI}(iv). 
Let $\gamma\in\mathcal{POPI}_{n}$ be such that $(\alpha, \gamma)\in\mathscr{L}^{\mathcal{POPI}_{n}}$ 
and $(\gamma, \beta)\in\mathscr{R}^{\mathcal{POPI}_{n}}$. 
Hence, by Lemma~\ref{GPOPI}(i), $\im(\gamma)=\im(\alpha)$ and so $\gamma\in\mathcal{POPI}_{n}(Y)$. 
On the other hand, by Lemma~\ref{GPOPI}(ii), $\dom(\gamma)=\dom(\beta)$, 
and so, by Theorem~\ref{green}, $(\gamma, \beta)\in\mathscr{R}^{\mathcal{POPI}_n(Y)}$. 
Thus, as $\beta$ is regular, $\gamma$ must be regular too, and so we have that both $\alpha$ and $\gamma$ are regular with $\im(\gamma)=\im(\alpha)$,
from which follows, by Theorem~\ref{green}, that $(\alpha, \gamma)\in\mathscr{L}^{\mathcal{POPI}_{n}(Y)}$. 
Therefore, $(\alpha, \beta)\in\mathscr{D}^{\mathcal{POPI}_{n}(Y)}$, as required. 
\end{proof}

\section{Size and an isomorphism theorem}\label{secc}

Let $Y$ be a subset of $\Omega_{n}$ of size $r$, with $1\leqslant r \leqslant n$. 
In this section, the cardinality of the semigroup $\mathcal{POPI}_{n}(Y)$ is determined, and an isomorphism theorem for semigroup of the type
$\mathcal{POPI}_{n}(Y)$ is given.

Notice that, as the elements of the semigroup $\mathcal{POPI}_{n}(Y)$ are injective orientation-preserving partial transformations, 
for any nonempty sets $A\subseteq \Omega_{n}$ and $B\subseteq Y$  with $|A|=|B|$, clearly, we have 
$$
|\{\alpha\in\mathcal{POPI}_{n}(Y)\,|\,\dom(\alpha)=A, \im(\alpha)=B\}|=|B|.
$$ 

Let $J_{k}=\{\alpha\in\mathcal{POPI}_{n}(Y)\,|\,\rank(\alpha)=k\}$ for $1\leqslant k\leqslant r$. Since the numbers of distinct domains contained in $\Omega_{n}$ of size $k$ and distinct images contained in $Y$ of size $k$ are $\tbinom{n}{k}$ and $\tbinom{r}{k}$, respectively,
it follows that $|J_{k}|=k\tbinom{n}{k}\tbinom{r}{k}$. As $\mathcal{POPI}_{n}(Y)$ also contains the empty transformation, we obtain 
\[
|\mathcal{POPI}_{n}(Y)|=1+\sum\limits ^{r}_{k=1}|J_{k}|=1+\sum\limits ^{r}_{k=1}k\tbinom{n}{k}\tbinom{r}{k}.
\]
Now, considering the Combinatorial Identity (3.30) in~\cite[p.25]{Gould:1972}, i.e. 
$$
\sum\limits ^{r}_{k=1}k\tbinom{r}{k}\tbinom{n}{k}=r\tbinom{n+r-1}{r},
$$
we conclude the following result.

\begin{proposition}
The cardinality of $\mathcal{POPI}_{n}(Y)$ is equal to $1+r\tbinom{n+r-1}{r}$.
\end{proposition}

\smallskip 

Let us denote by $E(S)$ the set of idempotents of a semigroup $S$. 

For any $\alpha\in\mathcal{PT}_{n}$, denote the set of fixed points of $\alpha$ by $\fix(\alpha)$, i.e.
$\fix(\alpha)=\{ x\in\dom(\alpha)\mid x\alpha=x\}$. 

\begin{lemma}\label{isoa}
Let $Y$ and $Z$ be nonempty subsets of $\Omega_{n}$. 
Let $\Phi: \mathcal{POPI}_{n}(Y)\longrightarrow \mathcal{POPI}_{n}(Z)$ be an isomorphism. 
Then, $\Phi$ induces a bijection $\varphi: Y\longrightarrow Z$ such that: 
\begin{enumerate}[(i)]
\item $\mathrm{id}_{\{y\}}\Phi=\mathrm{id}_{\{y\varphi\}}$ for any $y\in Y$;
\item for any $\alpha\in \mathcal{POPI}_{n}(Y)$, if $y\in Y\cap\dom(\alpha)$, then 
$y\varphi\in\dom(\alpha\Phi)$ and $(y\varphi)(\alpha\Phi)=(y\alpha)\varphi$; 
\item for any $\alpha\in \mathcal{POPI}_{n}(Y)$, $\fix(\alpha\Phi)=\fix(\alpha)\varphi$;
\item for any $\alpha\in\mathcal{POPI}_{n}(Y)$, $\im(\alpha\Phi)=\im(\alpha)\varphi$. 
\end{enumerate}
\end{lemma}

\begin{proof}
(i) The empty transformation $\varnothing$ is the zero of the semigroups $\mathcal{POPI}_{n}(Y)$ and $\mathcal{POPI}_{n}(Z)$. 
Since $\Phi$ is an isomorphism, we have $\varnothing\Phi=\varnothing$. 
For any $y\in Y$, since $\mathrm{id}_{\{y\}}$ is an idempotent, it follows that $\mathrm{id}_{\{y\}}\Phi$ is an idempotent, 
and so $\mathrm{id}_{\{y\}}\Phi=\mathrm{id}_{Z_{y}}$ for some nonempty subset $Z_{y}$ of $Z$. 
Clearly, $\mathrm{id}_{\{y_{1}\}}\mathrm{id}_{\{y_{2}\}}=\varnothing$ for any distinct $y_{1},\,y_{2}\in Y$, and so, 
it follows from $\mathrm{id}_{Z_{y_{1}}\cap Z_{y_{2}}}=\mathrm{id}_{Z_{y_{1}}}\mathrm{id}_{Z_{y_{2}}}=(\mathrm{id}_{\{y_{1}\}}\Phi)(\mathrm{id}_{\{y_{2}\}}\Phi)=
(\mathrm{id}_{\{y_{1}\}}\mathrm{id}_{\{y_{2}\}})\Phi=\varnothing\Phi=\varnothing$ that $Z_{y_{1}}\cap Z_{y_{2}}=\varnothing$. 
Therefore, $|Y|\leqslant \sum_{y\in Y}|Z_y| \leqslant |Z|$. 
Since $\Phi^{-1}$ is an isomorphism from $\mathcal{POPI}_{n}(Z)$ to $\mathcal{POPI}_{n}(Y)$, a similar argument can show that $|Z|\leqslant |Y|$, 
whence $|Y|=\sum_{y\in Y}|Z_y| =|Z|$, and so $Z_{y}$ is a singleton for each $y\in Y$, say $Z_{y}=\{z_y\}$. 
Moreover, the mapping $\varphi$: $Y\longrightarrow Z$ defined by $y\varphi =z_y$, for each $y \in Y$, is clearly a bijection. 

\smallskip 

(ii) Let $\alpha\in \mathcal{POPI}_{n}(Y)$ and take $y\in Y\cap\dom(\alpha)$. Then, 
$$
\mathrm{id}_{\{(y\alpha)\varphi\}}=\mathrm{id}_{\{y\alpha\}}\Phi=
(\transf{y\alpha \\ y }\transf{y \\ y\alpha})\Phi=
(\transf{y\alpha \\ y}\mathrm{id}_{\{y\}}\alpha)\Phi\\
=(\transf{y\alpha \\ y }\Phi)\mathrm{id}_{\{y\varphi\}}(\alpha\Phi)=
(\transf{y\alpha \\ y }\Phi)\transf{y\varphi \\ (y\varphi)(\alpha\Phi)},
$$
whence  $\{(y\alpha)\varphi\}={\im}(\mathrm{id}_{\{(y\alpha)\varphi\}})\subseteq\{(y\varphi)(\alpha\Phi)\}$. Consequently,
 $(y\alpha)\varphi=(y\varphi)(\alpha\Phi)$.

\smallskip 

(iii) Let $\alpha\in\mathcal{POPI}_{n}(Y)$. If $y\in \fix (\alpha)$, then $y\in Y\cap{\dom}(\alpha)$, and so by (ii),  $(y\varphi)(\alpha\Phi)=(y\alpha)\varphi=y\varphi$. Hence, $y\varphi\in \fix(\alpha\Phi)$ and so $\fix(\alpha)\varphi\subseteq \fix(\alpha\Phi)$. 
Since $\varphi$ is injective, it follows that $|\fix(\alpha)|=|\fix(\alpha)\varphi|\leqslant |\fix(\alpha\Phi)|$. 
Similarly, as $\Phi^{-1}$ is an isomorphism from $\mathcal{POPI}_{n}(Z)$ to $\mathcal{POPI}_{n}(Y)$, 
we have $|\fix(\alpha\Phi)|\leqslant|\fix((\alpha\Phi)\Phi^{-1}|=|\fix(\alpha)|$, whence $|\fix(\alpha)\varphi|=|\fix(\alpha\Phi)|$,  
and so $\fix(\alpha)\varphi=\fix(\alpha\Phi)$.  

\smallskip 

(iv) First, notice that if $\alpha\in E(\mathcal{POPI}_{n}(Y))$, then $\alpha\Phi\in E(\mathcal{POPI}_{n}(Z))$, whence ${\im}(\alpha)=\fix(\alpha)$ and ${\im}(\alpha\Phi)=\fix(\alpha\Phi)$. Thus, from (iii), it follows that ${\im}(\alpha\Phi)={\im}(\alpha)\varphi$. 

Now, let $\alpha$ be any element of $\mathcal{POPI}_{n}(Y)$. Then, clearly, $\alpha^{-1}\alpha\in E(\mathcal{POPI}_{n}(Y))$ and 
$\alpha\Phi=(\alpha(\alpha^{-1}\alpha))\Phi=\alpha\Phi (\alpha^{-1}\alpha)\Phi$, whence $\im(\alpha\Phi) \subseteq \im ((\alpha^{-1}\alpha)\Phi)$. 
On the other hand, we also have $(\alpha\Phi)^{-1}(\alpha\Phi)\in E(\mathcal{POPI}_{n}(Z))$. Therefore, there exists $\beta\in E(\mathcal{POPI}_{n}(Y))$ 
such that $\beta\Phi=(\alpha\Phi)^{-1}(\alpha\Phi)$. 
Since $(\alpha\beta)\Phi=(\alpha\Phi)(\beta\Phi) = \alpha\Phi((\alpha\Phi)^{-1}(\alpha\Phi)) = \alpha\Phi$, it follows that $\alpha\beta=\alpha$, 
whence $\im(\alpha)\subseteq\im(\beta)$ and so $(\im(\alpha))\varphi\subseteq(\im(\beta))\varphi$. 
Thus,
$$
\im(\alpha\Phi)\subseteq 
\im((\alpha^{-1}\alpha)\Phi)= (\im (\alpha^{-1}\alpha))\varphi = (\im(\alpha))\varphi\subseteq(\im(\beta))\varphi = \im(\beta\Phi)= \im ((\alpha\Phi)^{-1}(\alpha\Phi))\subseteq\im(\alpha\Phi),
$$
from which follows that $\im(\alpha\Phi)= \im(\alpha)\varphi$. 
\end{proof}

\smallskip 

Let us consider the following two permutations of $\Omega_{n}$ of order $n$ and $2$, respectively: 
\begin{equation}\label{gh}
g=\begin{pmatrix} 
1&2&\cdots&n-1&n\\
2&3&\cdots&n&1
\end{pmatrix} 
\quad\text{and}\quad  
h=\begin{pmatrix} 
1&2&\cdots&n-1&n\\
n&n-1&\cdots&2&1
\end{pmatrix}. 
\end{equation} 
Then, for $n\geqslant3$, $g$ and $h$ generate the well-known dihedral group $\mathcal{D}_{2n}$ of order $2n$ 
(considered as a subgroup of the symmetric group $\mathcal{S}_n$). In fact, we have 
$$
\mathcal{D}_{2n}=\langle g,h\mid g^n=h^2=1, gh=hg^{n-1}\rangle=\{1,g,g^2,\ldots,g^{n-1}, h,hg,hg^2,\ldots,hg^{n-1}\}. 
$$

Observe that, for $n\in\{1,2\}$, the dihedral group $\mathcal{D}_{2n}=\langle g,h\mid g^n=h^2=1, gh=hg^{n-1}\rangle$ of order $2n$ 
(also known as the \textit{Klein four-group} for $n=2$) cannot be considered as a subgroup of $\mathcal{S}_n$. 

Next, recall the following result from \cite{Fernandes&al:2016}:

\begin{lemma}\cite[Propositions 1.1]{Fernandes&al:2016}\label{D2n}
A partial injective transformation $\alpha$ is a restriction of a permutation in $\mathcal{D}_{2n}$ if and only if $|j\alpha-i\alpha|\in\{j-i,n-(j-i)\}$, for all $i,j\in\dom(\alpha)$ such that $i<j$.
\end{lemma}

Now, we can prove the following isomorphism theorem. 

\begin{theorem}\label{isob}
Let $Y$ and $Z$ be nonempty subsets of $\Omega_{n}$. Then, the semigroups $\mathcal{POPI}_{n}(Y)$ and 
$\mathcal{POPI}_{n}(Z)$ are isomorphic if and only if 
$|Y|=|Z|\leqslant2$ or $Y\delta=Z$ for some $\delta\in\mathcal{D}_{2n}$.
\end{theorem}

\begin{proof}
If $Y\delta=Z$ for some $\delta\in\mathcal{D}_{2n}$, then it is routine to verify that the mapping $\Phi$: $\mathcal{POPI}_{n}(Y)\longrightarrow\mathcal{POPI}_{n}(Z)$ defined by $\alpha\Phi=\delta^{-1}\alpha\delta$, for all $\alpha\in\mathcal{POPI}_{n}(Y)$, is an isomorphism. 
On the other hand, suppose that $|Y|=|Z|\leqslant2$ and let $\sigma\in\mathcal{S}_{n}$ be such that $Y\sigma=Z$. 
Since any transformation of $\mathcal{I}_{n}$ with rank less than or equal to $2$ is orientation-preserving, we have $\sigma^{-1}\mathcal{POPI}_{n}(Y)\sigma\subseteq\mathcal{POPI}_{n}(Z)$ and so, as above, it is easy to check that the mapping $\Phi$: $\mathcal{POPI}_{n}(Y)\longrightarrow\mathcal{POPI}_{n}(Z)$ defined by $\alpha\Phi=\sigma^{-1}\alpha\sigma$ for all $\alpha\in\mathcal{POPI}_{n}(Y)$ is an isomorphism. 

Conversely, suppose that $\Phi:\mathcal{POPI}_{n}(Y)\longrightarrow\mathcal{POPI}_{n}(Z)$ is an isomorphism and 
let $\varphi:Y\longrightarrow Z$ be the bijection induced by $\Phi$ given by Lemma~\ref{isoa}. 
Then, $|Y|=|Z|$. Hence, suppose that $|Y|=|Z|\geqslant3$. 

Let $i,j\in Y$ be such that $i<j$ and take $k\in Y\setminus\{i,j\}$. 
Define 
$$
A(i,j,k) =\left\{  \begin{pmatrix} 
i&j&y\\
i&j&k
\end{pmatrix}\in \mathcal{POPI}_{n}(Y)\mid y\in\Omega_{n}
\right\}. 
$$
Then, it is easy to verify that 
$
A(i,j,k) =\left\{  \left(\begin{smallmatrix} 
i&j&y\\
i&j&k
\end{smallmatrix}\right)\mid y\in \{1,\ldots,i-1,j+1,\ldots,n\} 
\right\}, 
$
if $i<j<k$ or $k<i<j$, and 
$
A(i,j,k) =\left\{  \left(\begin{smallmatrix} 
i&j&y\\
i&j&k
\end{smallmatrix}\right)\mid y\in \{i+1,\ldots,j-1\} 
\right\}, 
$
if $i<k<j$. 
Thus, $|A(i,j,k)|\in\{n-(j-i)-1,j-i-1\}$.  
On the other hand, by Lemma \ref{isoa} (ii) and (iv), we have 
$$
A(i,j,k)\Phi =\left\{  \begin{pmatrix} 
i\varphi&j\varphi&z\\
i\varphi&j\varphi&k\varphi
\end{pmatrix}\in \mathcal{POPI}_{n}(Z)\mid z\in \Omega_{n}
\right\}
$$
and so, similarly, $|A(i,j,k)\Phi|\in\{n-|j\varphi-i\varphi|-1,|j\varphi-i\varphi|-1\}$. 
Therefore, $|j\varphi-i\varphi|\in \{j-i,n-(j-i)\}$. 

Now, by Lemma \ref{D2n}, it follows that $\varphi$ is a restriction of some permutation $\delta$ in  $\mathcal{D}_{2n}$, 
for which we have $Y\delta=Z$, as required. 
\end{proof}

\section{Rank of $\mathcal{POPI}_{n}(Y)$}\label{secd}

In this section, we determine the rank of the semigroup $\mathcal{POPI}_{n}(Y)$. Notice that, $\mathcal{POPI}_{n}(\Omega_{n})=\mathcal{POPI}_{n}$ and it is well known that the semigroup $\mathcal{POPI}_{n}$ has rank $2$ (see~\cite{Fernandes:2000}). 
Therefore, in what follows, we suppose that $Y=\{y_{1}<y_{2}<\cdots<y_{r}\}$ is a proper subset of $\Omega_{n}$.

Let $J_{m}^{S}=\{\alpha\in S\mid\rank(\alpha)= m\}$, for any subsemigroup $S$ of $\mathcal{PT}_{n}$ and $0\leqslant m\leqslant n$. 

We begin by recalling the following lemma. 

\begin{lemma}\cite[Lemma 5.2]{Li&al:2023}\label{POI}
For $0\leqslant m\leqslant r-2$, $J_{m}^{\mathcal{POI}_{n}(Y)}\subseteq\langle J_{m+1}^{\mathcal{POI}_{n}(Y)}\rangle$.
\end{lemma}

Next, observe that the permutation $g$ of $\Omega_{n}$ defined in (\ref{gh}) is an orientation-preserving transformation. 

We will now prove a series of lemmas. 

\begin{lemma}\label{POPIa}
For $0\leqslant m\leqslant r-2$, $J_{m}^{\mathcal{POPI}_{n}(Y)}\subseteq\langle J_{m+1}^{\mathcal{POPI}_{n}(Y)}\rangle$.
\end{lemma}

\begin{proof}
Let $\alpha\in\mathcal{POPI}_{n}(Y)$ be such that ${\rank}(\alpha)=m$. Then, by \cite[Proposition 3.1]{Fernandes:2000}, there exist $\alpha_{1}\in\mathcal{POI}_{n}$ and $\ell\in\{0,1,\cdots,n-1\}$ such that $\alpha=g^{\ell}\alpha_{1}$, where $g$ is the permutation defined in (\ref{gh}). Then, 
${\im}(\alpha_{1})={\im}(\alpha)\subseteq Y$, and so $\alpha_{1}\in\mathcal{POI}_{n}(Y)$. By Lemma~\ref{POI}, 
there exist $\beta,\gamma\in\mathcal{POI}_{n}(Y)$ with rank $m+1$ such that 
$\alpha_{1}=\beta\gamma$. 
Let $\delta=g^{\ell}\beta$. Then, ${\im}(\delta)={\im}(\beta)\subseteq Y$, 
and so $\delta\in\mathcal{POPI}_{n}(Y)$ with ${\rank}(\delta)=m+1$. Thus, $\alpha=\delta\gamma\in \langle J_{m+1}^{\mathcal{POPI}_{n}(Y)}\rangle$, as required.
\end{proof}

Let $V=\{\alpha\in\mathcal{POI}_{n}(Y)\mid \mbox{${\dom}(\alpha)\subseteq Y$ and ${\rank}(\alpha)=r-1$}\}$. 
Obviouly, we also have $V\subseteq\mathcal{POPI}_{n}(Y)$.

\begin{lemma}\label{POPIb}
$J_{r-1}^{\mathcal{POPI}_{n}(Y)}\subseteq\langle J_{r}^{\mathcal{POPI}_{n}(Y)}\cup V\rangle$.
\end{lemma}

\begin{proof}
Let $\alpha\in J_{r-1}^{\mathcal{POPI}_{n}(Y)}$. By \cite[Proposition 3.1]{Fernandes:2000}, there exist $\alpha_{1}\in\mathcal{POI}_{n}$ and $\ell\in\{0,1,\cdots,n-1\}$ such that $\alpha=g^{\ell}\alpha_{1}$. Then, ${\im}(\alpha_{1})={\im}(\alpha)\subseteq Y$, as $g$ is a permutation of $\Omega_{n}$, whence $\alpha_{1}\in J_{r-1}^{\mathcal{POI}_{n}(Y)}$. Assume that 
\[
\alpha_{1}= \begin{pmatrix}a_{1} &\cdots &a_{i-1} &a_{i} &a_{i+1} &\cdots &a_{r-1}\\y_{1} &\cdots &y_{i-1} &y_{i+1} &y_{i+2} &\cdots &y_{r} \end{pmatrix}, 
\]
with ${\im}(\alpha)=Y\backslash\{y_{i}\}$ for some $1\leqslant i\leqslant r$. Since $\rank(\alpha)=r-1<n$, then 
there exist $c\in \Omega_n\backslash\{a_{1},\ldots,a_{r-1}\}$ and $1\leqslant j\leqslant r$ 
such that $a_{j-1}<c<a_{j}$, where $a_{0}=0$ and $a_{r}=n+1$. 
Let
$$
\beta = \begin{pmatrix}a_{1} &\cdots &a_{i-1} &a_{i} &\cdots &a_{j-1} &c &a_{j} &\cdots &a_{r-1} \\y_{1} &\cdots &y_{i-1} &y_{i} &\cdots &y_{j-1} &y_{j} &y_{j+1} &\cdots &y_{r} \end{pmatrix}
$$
and
$$
\gamma = \begin{pmatrix}y_{1} &\cdots &y_{i-1} &y_{i} &\cdots &y_{j-1} &y_{j+1} &\cdots &y_{r} \\y_{1} &\cdots &y_{i-1} &y_{i+1} &\cdots &y_{j} &y_{j+1} &\cdots &y_{r} \end{pmatrix},
$$
if $i\leqslant j$; and
$$
\beta =\begin{pmatrix}a_{1} &\cdots &a_{j-1} &c &a_{j} &\cdots &a_{i-1} &a_{i} &\cdots &a_{r-1} \\y_{1} &\cdots &y_{j-1} &y_{j} &y_{j+1} &\cdots &y_{i} &y_{i+1} &\cdots &y_{r} \end{pmatrix}
$$
and
$$
\gamma =\begin{pmatrix}y_{1} &\cdots &y_{j-1} &y_{j+1} &\cdots &y_{i} &y_{i+1} &\cdots &y_{r} \\y_{1} &\cdots &y_{j-1} &y_{j} &\cdots &y_{i-1} &y_{i+1} &\cdots &y_{r} \end{pmatrix}, 
$$ 
if $i>j$. Then, it is a routine matter to verify that $\beta\in J_{r}^{\mathcal{POPI}_{n}(Y)}$, $\gamma\in V$ and $\alpha_{1}=\beta\gamma$. 
Hence, $\alpha=g^{\ell}\beta\gamma$. Since $g^{\ell}\beta\in J_{r}^{\mathcal{POPI}_{n}(Y)}$, we have $\alpha\in\langle J_{r}^{\mathcal{POPI}_{n}(Y)}\cup V\rangle$, as required.
\end{proof}

\begin{lemma}\label{POPIc}
$V\subseteq\langle J_{r}^{\mathcal{POPI}_{n}(Y)}\rangle$.
\end{lemma}

\begin{proof}
Let $\alpha\in V$ and suppose that ${\dom}(\alpha)=Y\backslash\{y_{i}\}$ and ${\im}(\alpha)=Y\backslash\{y_{j}\}$, 
for some $1\leqslant i, j\leqslant r$. Clearly, if $i\geqslant j$, then
\[
\alpha=\begin{pmatrix}y_{1} &\cdots &y_{j-1} &y_{j} &\cdots &y_{i-1} &y_{i+1} &\cdots &y_{r}\\y_{1} &\cdots &y_{j-1} &y_{j+1} &\cdots &y_{i} &y_{i+1} &\cdots &y_{r}\end{pmatrix};
\]
if $i<j$, then
\[
\alpha=\begin{pmatrix}y_{1} &\cdots &y_{i-1} &y_{i+1} &\cdots &y_{j} &y_{j+1} &\cdots &y_{r}\\y_{1} &\cdots &y_{i-1} &y_{i} &\cdots &y_{j-1} &y_{j+1} &\cdots &y_{r}\end{pmatrix}.
\]
Now, we distinguish three major cases.

\smallskip 

\noindent{\sc case~1.} $y_{1}>1$. If $i\geqslant j$, then take 
$$
\beta =\begin{pmatrix}y_{1} &\cdots &y_{j-1} &y_{j} &\cdots &y_{i-1} &y_{i} &\cdots &y_{r}\\y_{r-j+2} &\cdots &y_{r} &y_{1} &\cdots &y_{i-j} &y_{i-j+1} &\cdots &y_{r-j+1}\end{pmatrix}
$$
and
$$
\gamma =\begin{pmatrix}1 &y_{1} &\cdots &y_{i-j} &y_{i-j+2} &\cdots &y_{r-j+1} &y_{r-j+2} &\cdots &y_{r}\\y_{j} &y_{j+1} &\cdots &y_{i} &y_{i+1} &\cdots &y_{r} &y_{1} &\cdots &y_{j-1}\end{pmatrix}; 
$$
if $i<j$, then take 
$$
\beta =\begin{pmatrix}y_{1} &\cdots &y_{i-1} &y_{i} &y_{i+1} &\cdots &y_{j} &y_{j+1} &\cdots &y_{r}\\y_{r-j+1} &\cdots &y_{r-j+i-1} &y_{r-j+i} &y_{r-j+i+1} &\cdots &y_{r} &y_{1} &\cdots &y_{r-j}\end{pmatrix}
$$
and
$$
\gamma =\begin{pmatrix}1 &y_{1} &\cdots &y_{r-j} &y_{r-j+1} &\cdots &y_{r-j+i-1} &y_{r-j+i+1} &\cdots &y_{r}\\y_{j} &y_{j+1} &\cdots &y_{r} &y_{1} &\cdots &y_{i-1} &y_{i} &\cdots &y_{j-1}\end{pmatrix}.
$$

\smallskip 

\noindent{\sc case~2.} $y_{r}<n$. If $i\geqslant j$, then take 
$$
\beta=\begin{pmatrix}y_{1} &\cdots &y_{j-1} &y_{j} &\cdots &y_{i-1} &y_{i} &\cdots &y_{r}\\y_{r-j+2} &\cdots &y_{r} &y_{1} &\cdots &y_{i-j} &y_{i-j+1} &\cdots &y_{r-j+1}\end{pmatrix}
$$
and
$$
\gamma =\begin{pmatrix}y_{1} &\cdots &y_{i-j} &y_{i-j+2} &\cdots &y_{r-j+1} &y_{r-j+2} &\cdots &y_{r} &n\\y_{j+1} &\cdots &y_{i} &y_{i+1} &\cdots &y_{r} &y_{1} &\cdots &y_{j-1} &y_{j}\end{pmatrix}; 
$$
if $i<j$, then take 
$$
\beta =\begin{pmatrix}y_{1} &\cdots &y_{i-1} &y_{i} &y_{i+1} &\cdots &y_{j} &y_{j+1} &\cdots &y_{r}\\y_{r-j+1} &\cdots &y_{r-j+i-1} &y_{r-j+i} &y_{r-j+i+1} &\cdots &y_{r} &y_{1} &\cdots &y_{r-j}\end{pmatrix}
$$
and
$$
\gamma =\begin{pmatrix}y_{1} &\cdots &y_{r-j} &y_{r-j+1} &\cdots &y_{r-j+i-1} &y_{r-j+i+1} &\cdots &y_{r} &n\\y_{j+1} &\cdots &y_{r} &y_{1} &\cdots &y_{i-1} &y_{i} &\cdots &y_{j-1} &y_{j}\end{pmatrix}.
$$

\smallskip 

\noindent{\sc case~3.} $y_{1}=1$, $y_{r}=n$ and $y_{k}<y_{k+1}-1$, for some $1\leqslant k\leqslant r-1$. 

\smallskip 

Subcase 3.1. $i\geqslant j$. If $k\geqslant j-1$, then take 
$$
\addtocounter{MaxMatrixCols}{3} 
\beta =\begin{pmatrix}y_{1} &\cdots &y_{j-1} &y_{j} &\cdots &y_{i-1} &y_{i} &y_{i+1} &\cdots &y_{r+j-k-1} &y_{r+j-k} &\cdots &y_{r}\\y_{k+2-j} &\cdots &y_{k} &y_{k+1} &\cdots &y_{i+k-j} &y_{i+k+1-j} &y_{i+k+2-j} &\cdots &y_{r} &y_{1} &\cdots &y_{k+1-j}\end{pmatrix}
$$
and
$$
\gamma  =\begin{pmatrix}y_{1} &\cdots &y_{k+1-j} &y_{k+2-j} &\cdots &y_{k} &y_{k}+1 &y_{k+1} &\cdots &y_{i+k-j} &y_{i+k+2-j} &\cdots &y_{r}\\y_{r+j-k} &\cdots &y_{ r} &y_{1} &\cdots &y_{j-1} &y_{j} &y_{j+1} &\cdots &y_{i} &y_{i+1} &\cdots &y_{r+j-k-1}\end{pmatrix};
$$
if $k<j-1$, then take 
$$
\addtocounter{MaxMatrixCols}{3}
\beta =\begin{pmatrix}y_{1} &\cdots &y_{j-k-1} &y_{j-k} &\cdots &y_{j-1} &y_{j} &\cdots &y_{i-1} &y_{i} &y_{i+1} &\cdots &y_{r}\\y_{r+k+2-j} &\cdots &y_{r} &y_{1} &\cdots &y_{k} &y_{k+1} &\cdots &y_{i+k-j} &y_{i+k+1-j} &y_{i+k+2-j} &\cdots &y_{r+k+1-j}\end{pmatrix}
$$
and
$$
\gamma =\begin{pmatrix}y_{1} &\cdots &y_{k} &y_{k}+1 &y_{k+1} &\cdots &y_{i+k-j} &y_{i+k+2-j} &\cdots &y_{r+k+1-j} &y_{r+k+2-j} &\cdots &y_{r}\\y_{j-k} &\cdots &y_{j-1} &y_{j} &y_{j+1} &\cdots &y_{i} &y_{i+1} &\cdots &y_{r} &y_{1} &\cdots &y_{j-k-1}\end{pmatrix}.
$$

\smallskip 

Subcase 3.2. $i<j$. If $k\geqslant j$, then take 
$$
\addtocounter{MaxMatrixCols}{3}
\beta=\begin{pmatrix}y_{1} &\cdots &y_{i-1} &y_{i} &y_{i+1} &\cdots &y_{j} &y_{j+1} &\cdots &y_{r+j-k} &y_{r+j+1-k} &\cdots &y_{r}\\y_{k+1-j} &\cdots &y_{k+i-j-1} &y_{k+i-j} &y_{k+i+1-j} &\cdots &y_{k} &y_{k+1} &\cdots &y_{r} &y_{1} &\cdots &y_{k-j}\end{pmatrix}
$$
and
$$
\gamma =\begin{pmatrix}y_{1} &\cdots &y_{k-j} &y_{k+1-j} &\cdots &y_{k+i-j-1} &y_{k+i+1-j} &\cdots &y_{k} &y_{k}+1 &y_{k+1} &\cdots &y_{r}\\y_{r+j+1-k} &\cdots &y_{ r} &y_{1} &\cdots &y_{i-1} &y_{i} &\cdots &y_{j-1} &y_{j} &y_{j+1}&\cdots &y_{r+j-k}\end{pmatrix};
$$
if $j+1-i\leqslant k<j$, then take 
$$
\addtocounter{MaxMatrixCols}{3}
\beta =\begin{pmatrix}y_{1} &\cdots &y_{j-k} &y_{j+1-k} &\cdots &y_{i-1} &y_{i} &y_{i+1} &\cdots &y_{j} &y_{j+1} &\cdots &y_{r}\\y_{r+k+1-j} &\cdots &y_{r} &y_{1} &\cdots &y_{i+k-j-1} &y_{i+k-j} &y_{i+k+1-j} &\cdots &y_{k} &y_{k+1} &\cdots &y_{r+k-j}\end{pmatrix}
$$
and
$$
\gamma =\begin{pmatrix}y_{1} &\cdots &y_{i+k-j-1} &y_{i+k+1-j} &\cdots &y_{k} &y_{k}+1 &y_{k+1} &\cdots &y_{r+k-j} &y_{r+k+1-j} &\cdots &y_{r}\\y_{j+1-k} &\cdots &y_{ i-1} &y_{i} &\cdots &y_{j-1} &y_{j} &y_{j+1} &\cdots &y_{r} &y_{1} &\cdots &y_{j-k}\end{pmatrix};
$$
if $k<j+1-i$, then take 
$$
\addtocounter{MaxMatrixCols}{3}
\beta =\begin{pmatrix}y_{1} &\cdots &y_{i-1} &y_{i} &\cdots &y_{j-k} &y_{j+1-k} &\cdots &y_{j} &y_{j+1} &\cdots &y_{r}
\\y_{r+k+1-j} &\cdots &y_{r+i+k-j-1} &y_{r+i+k-j} &\cdots &y_{r} &y_{1} &\cdots &y_{k} &y_{k+1} &\cdots &y_{r+k-j}\end{pmatrix}
$$
and
$$
\gamma =\begin{pmatrix}y_{1} &\cdots &y_{k} &y_{k}+1 &y_{k+1} &\cdots &y_{r+k-j} &y_{r+k+1-j} &\cdots &y_{r+i+k-j-1} &y_{r+i+k+1-j} &\cdots &y_{r}\\y_{j-k} &\cdots &y_{j-1} &y_{j} &y_{j+1} &\cdots &y_{r} &y_{1} &\cdots &y_{i-1} &y_{i} &\cdots &y_{j-k-1}\end{pmatrix}.
$$

In all cases, it is a routine matter to verify that $\beta,\gamma\in J_{r}^{\mathcal{POPI}_{n}(Y)}$ and $\alpha=\beta\gamma$. 
Therefore, $\alpha\in \langle J_{r}^{\mathcal{POPI}_{n}(Y)}\rangle$, as required.
\end{proof}

From Lemmas~\ref{POPIa}, \ref{POPIb} and \ref{POPIc}, 
it follows that $\mathcal{POPI}_{n}(Y)$ can be generated by its elements of rank $r$, i.e. the following result holds.

\begin{lemma}\label{}
$\mathcal{POPI}_{n}(Y)=\langle J_{r}^{\mathcal{POPI}_{n}(Y)}\rangle$.
\end{lemma}

\smallskip 

Now, let
\[
\bar{g}=\begin{pmatrix}y_{1} &y_{2} &\cdots &y_{r-1} &y_{r}\\y_{2} &y_{3} &\cdots &y_{r} &y_{1}\end{pmatrix}\in\mathcal{POPI}_{n}(Y).
\]

Then, we have the following result.

\begin{lemma}\label{E}
Let $\alpha,\beta\in \mathcal{POPI}_{n}(Y)$ be such that ${\rank}(\alpha)={\rank}(\beta)=r$ and $\dom(\alpha)=\dom(\beta)$. 
Then, $\alpha=\beta \bar{g}^{t}$ for some $0\leqslant t\leqslant r-1$.
\end{lemma}

\begin{proof}
Assume that ${\dom}(\alpha)={\dom}(\beta)=\{a_{1}<a_{2}<\cdots<a_{r}\}$. Then,
$$
\alpha =\begin{pmatrix}a_{1} &a_{2} &\cdots &a_{i-1} &a_{i} &a_{i+1} &\cdots &a_{r}\\y_{r-i+2} &y_{r-i+3} &\cdots &y_{r} &y_{1} &y_{2} &\cdots &y_{r-i+1}\end{pmatrix}
$$
and
$$
\beta =\begin{pmatrix}a_{1} &a_{2} &\cdots &a_{j-1} &a_{j} &a_{j+1} &\cdots &a_{r}\\y_{r-j+2} &y_{r-j+3} &\cdots &y_{r} &y_{1} &y_{2} &\cdots &y_{r-j+1}\end{pmatrix}
$$
for some $1\leqslant i,j\leqslant r$. Let
\begin{equation*}
t=
\begin{cases}
j-i &\mbox{if \,$i\leqslant j$},\\
r-i+j &\mbox{if \,$i>j$}.
\end{cases}
\end{equation*}
So, it is routine to show that $\alpha=\beta \bar{g}^{t}$, as required. 
\end{proof}

\begin{lemma}\label{F}
Let $\alpha,\beta\in \mathcal{POPI}_{n}(Y)$ be two elements of rank $r$ such that ${\rank}(\alpha\beta)=r$. Then, $(\alpha\beta, \alpha)\in\mathscr{R}$.
\end{lemma}

\begin{proof}
Since $\dom(\alpha\beta)\subseteq\dom(\alpha)$ and $\rank(\alpha)=\rank(\alpha\beta)$, 
it follows that $\dom(\alpha\beta)=\dom(\alpha)$, whence $(\alpha\beta, \alpha)\in\mathscr{R}$, by Theorem~\ref{green}(ii).
\end{proof}

An immediate consequence of Lemmas~\ref{E} and~\ref{F} is that any generating set of the semigroup $\mathcal{POPI}_{n}(Y)$ must contain at least one element from each $\mathscr{R}$-class of rank $r$ of $\mathcal{POPI}_{n}(Y)$, and $\mathcal{POPI}_{n}(Y)$ can be generated by any set containing $\bar{g}$ as well as at least one element from each $\mathscr{R}$-class of rank $r$ of $\mathcal{POPI}_{n}(Y)$. On the other hand, the number of $\mathscr{R}$-class of rank $r$ of $\mathcal{POPI}_{n}(Y)$ is $\tbinom{n}{r}$. Thus, we can now present the main result of this section as follows.

\begin{theorem}
Let $Y$ be a nonempty subset of $\Omega_{n}$ of size $r$. Then
\begin{equation*}
 {\rank}(\mathcal{POPI}_{n}(Y))=
\begin{cases}
\tbinom{n}{r} &\mbox{if \,$1\leqslant |Y|\leqslant n-1$},\\
\,\,2 &\mbox{if \,$|Y|=n$}.
\end{cases}
\end{equation*}
\end{theorem}

\lastpage 
\end{document}